\documentclass[reqno]{amsart}
\newcommand\version{August 27, 2023}
\usepackage{amsmath,mathtools,amssymb}

\usepackage{amsthm}
\usepackage{amsfonts}
\usepackage[shortlabels]{enumitem}
\usepackage{color}
\usepackage{hyperref}

\newtheorem{theorem}{Theorem}
\newtheorem{proposition}[theorem]{Proposition}

\newtheorem{conjecture}{Conjecture}
\newtheorem*{conjecture*}{Conjecture}

\theoremstyle{definition}

\newcommand{\N}{\mathbb{N}}
\newcommand{\Z}{\mathbb{Z}}

\newcommand{\R}{\mathbb{R}}
\newcommand{\C}{\mathbb{C}}
\newcommand{\E}{\mathbf{E}}
\newcommand{\bs}{\mathbb{S}}

\newcommand\e{\mathrm{e}}
\newcommand\I{\mathrm{i}}

\newcommand{\F}{\mathcal{F}}
\newcommand\re{\operatorname{Re}}
\newcommand\im{\operatorname{Im}}

\newcommand{\ce}{\mathcal{E}}
\newcommand{\cv}{\mathcal{V}}
\newcommand{\cs}{\mathcal{S}}
\newcommand{\ch}{\mathcal{H}}

\newcommand\eps\varepsilon
\renewcommand\epsilon\varepsilon
\renewcommand\rho\varrho
\newcommand\lm\lambda

\newcommand{\supp}{\operatorname{supp}}
\newcommand{\dist}{\operatorname{dist}}

\newcommand{\rd}{d}

\newcommand{\me}[1]{\mathrm{e}^{#1}}
\newcommand{\one}{\mathbf{1}}

\newcommand{\beq}{\begin{equation}}
\newcommand{\eeq}{\end{equation}}
\newcommand{\be}{\begin{equation*}}
\newcommand{\ee}{\end{equation*}}
\newcommand{\bmat}{\begin{pmatrix}}
\newcommand{\emat}{\end{pmatrix}}

\begin{document}
\title[Eigenvalues of random Schr\"odinger operators --- \version]{Lieb--Thirring-type inequalities for random Schr\"odinger operators with complex potentials}

\author[J.-C.\ Cuenin]{Jean-Claude Cuenin}
\address[Jean-Claude Cuenin]{Department of Mathematical Sciences, Loughborough University, Loughborough, Leicestershire LE11 3TU, United Kingdom}
\email{J.Cuenin@lboro.ac.uk}

\author[K. Merz]{Konstantin Merz}
\address[Konstantin Merz]{Institut f\"ur Analysis und Algebra, Technische Universit\"at Braunschweig, Universit\"atsplatz 2, 38106 Braun\-schweig, Germany, and Department of Mathematics, Graduate School of Science, Osaka University, Toyonaka, Osaka 560-0043, Japan}
\email{k.merz@tu-bs.de}

\date{\version}

\begin{abstract}
  We review some results and proofs on eigenvalue bounds for random Schr\"odinger operators with complex-valued potentials. We also include new Schatten norm estimates for the resolvent and use them to obtain bounds for sums of eigenvalues.
\end{abstract}

\thanks{
  The second-named author thanks Neal Bez and Yutaka Terasawa for the organization and invitation to the RIMS Symposium on Harmonic Analysis and Nonlinear Partial Differential Equations, where parts of this note were presented.
  This work was supported by the Research Institute for Mathematical Sciences, an International Joint Usage/Research Center located in Kyoto University.
  Support by the Engineering \& Physical Sciences Research Council [grant numberEP/X011488/1] (J.-C. C.) and by the PRIME programme of the German Academic Exchange Service (DAAD) with funds from the German Federal Ministry of Education and Research (BMBF) (K.M.) is acknowledged.
  We thank Haruya Mizutani for helpful discussions.}

\maketitle
\tableofcontents

\section{Introduction}

\subsection{Eigenvalue estimates for Schr\"odinger operators with complex potentials}
Schr\"odinger operators with complex-valued potentials arise in the analysis of scattering resonances, open or damped quantum systems, etc. For decaying potentials $V:\R^d\to\C$ that are not too singular, $-\Delta-V$ can be realized as an m-sectorial operator in $L^2(\R^d)$, with essential spectrum $[0,\infty)$; see, e.g., \cite[Proposition~B.2]{Frank2018E}\footnote{In the present situation, all definitions of essential spectrum coincide.}.
We are interested in the discrete eigenvalues of finite algebraic multiplicity, $\sigma_{\rm disc}(-\Delta-V)$, which can only accumulate at $[0,\infty)$. Natural quantities of interest are, e.g., the number and distribution of eigenvalues of $-\Delta-V$. For complex-valued $V$, the analysis of these quantities yields results that are surprisingly different from those in the case of real-valued $V$. We refer to \cite[Section~5.13]{Frank2021} and the references therein for a survey.

In the first part of this note, we review estimates for single eigenvalues of $-\Delta-V$ which only depend on $\|V\|_q$.
In the second part, we prove new estimates for sums of eigenvalues. 
As in previous studies on this subject, the decisive feature of $V$ is its decay.
Hence, for the rest of the paper, we will always assume that $V$ is bounded.

To motivate the sought-after estimates, let us consider the situation of real-valued $V$ for a moment. In this case, the Keller--Lieb--Thirring inequality
\begin{align}
  \label{eq:classiclt}
  \sum_{\lambda_j\in\sigma_{\rm disc}(-\Delta-V)}|\lambda_j|^{q-d/2} \lesssim_{d,q} \int_{\R^d}V(x)_+^q\,dx, \quad
  \begin{cases}
    q\in[1,\infty), & \ d=1 \\
    q\in(1,\infty), & \ d=2 \\
    q\in[d/2,\infty), & \ d\geq3
  \end{cases}
\end{align}
quantifies the eigenvalue accumulation at zero when $q>d/2$ and the number of eigenvalues when $q=d/2$ and $d\geq3$. Here and below, we write $A\lesssim B$ for two non-negative quantities $A,B\geq 0$ to indicate that there is a constant $C>0$ such that $A\leq C B$. The dependence of the constant on fixed parameters like $d$ and $q$ is usually omitted.
The notation $A\sim B$ means $A\lesssim B\lesssim A$.
Inequality \eqref{eq:classiclt} for single eigenvalues was first proved by Keller \cite{Keller1961} and later generalized to eigenvalue sums by Lieb and Thirring \cite{LiebThirring1976}, who used this inequality to give a simple proof of stability of quantum matter.
Inequality~\eqref{eq:classiclt} is interesting because it is invariant under spatial rescalings and recovers Weyl's law in the strong-coupling limit, i.e., when one replaces $V$ with $\kappa V$, and considers the limit $\kappa\to\infty$.

A natural question to ask is how much of \eqref{eq:classiclt} survives when $V$ is complex-valued. A first answer in the context of single eigenvalues was given by Abramov, Aslanyan, and Davies \cite{Abramovetal2001}. For $d=1$, they showed
\begin{align}
  |z|^{1/2} \leq \frac12 \int_\R|V|\,dx, \quad z\in\sigma_{\rm disc}(-\Delta-V).
\end{align}
Shortly afterwards, Frank, Laptev, Lieb, and Seiringer \cite{Franketal2006} considered the higher-dimensional case and $V\in L^q(\R^d)$, $q\geq d/2$, and showed
\begin{align}
  \label{eq:flls}
  \sum_{z_j\in\sigma_{\rm disc}(-\Delta-V),\,|\im z_j|\geq\kappa\re z_j} |z_j|^{q-\frac d2}
  \lesssim \left(1+\kappa^{-1}\right)^{q} \int_{\R^d}|V|^q.
\end{align}
While this inequality shows that \eqref{eq:classiclt} continues to hold for eigenvalues of $-\Delta-V$ outside any fixed sector in the complex plane, \eqref{eq:flls} deteriorates for eigenvalues close to $[0,\infty)$.
For individual eigenvalues, Laptev and Safronov \cite[p.~31]{LaptevSafronov2009} made the following

\begin{conjecture}
  \label{ls}
  Let $d\geq2$, $q\in(d/2,d]$. Then, any eigenvalue $z\in\C\setminus[0,\infty)$ of $-\Delta-V$ satisfies
  \begin{align}
    \label{eq:ls}
    |z|^{q-\frac{d}{2}} \lesssim_{d,q} \int_{\R^d}|V|^{q}.
  \end{align}
\end{conjecture}

Laptev and Safronov were careful to state this conjecture only for $q\leq d$ because they were aware of the real-valued Wigner--Neumann potential \cite[p.~223]{ReedSimon1978}. This is a radial potential which oscillates and decays like $|x|^{-1}$, i.e., it almost belongs to $L^d(\R^d)$, and generates an eigenvalue embedded at $1$.
%
B\"ogli \cite{Bogli2017} proved that the range of $q$ would be best possible for radial potentials. Frank and Simon \cite{FrankSimon2017} proved Conjecture~\ref{ls} for radial potentials when $q<d$. Disregarding the endpoint $q=d$, these two results provide a complete answer to the Laptev--Safronov conjecture for radial potentials.
For general potentials, Frank \cite{Frank2011} proved
\begin{align}
  \label{eq:franklt}
  |z|^{q-d/2} \lesssim \int_{\R^d}|V(x)|^q\,dx, \quad z\in\sigma_{\rm disc}(-\Delta-V), \ q\leq\frac{d+1}{2}.
\end{align}
The starting point to prove \eqref{eq:franklt} is the Birman--Schwinger principle, i.e., the statement that
\begin{align}
  z\in\sigma_{\rm disc}(-\Delta-V)
  \Leftrightarrow
  1 \in \sigma_{\rm disc}(BS(z)),
  \quad BS(z):=|V|^{1/2}(-\Delta-z)^{-1}V^{1/2},
\end{align}
with $V^{1/2}=V/|V|^{1/2}$.
In particular, if $z\in\sigma_{\rm disc}(-\Delta-V)$, then the spectral radius ${\rm spr}(BS(z))\geq1$. Thus, to get an eigenvalue estimate, it suffices to bound ${\rm spr}(BS(z))$ from above by a negative power of $|z|$. A simple estimate is ${\rm spr}(BS(z))\leq\|BS(z)\|$.
By H\"older's inequality and the homogeneity of $-\Delta$,
\begin{align}
  \begin{split}
    1
    \leq {\rm spr}(BS(z))\leq\|BS(z)\|
    & \leq \|V\|_q \|(-\Delta-z)^{-1}\|_{p\to p'} \\
    & = |z|^{d/(2q)-1} \|V\|_q \|(-\Delta-z/|z|)^{-1}\|_{p\to p'},
  \end{split}
\end{align}
where $1/q=1/p-1/p'$ and $p'=(1-1/p)^{-1}$ for $1\leq p\leq\infty$.
Kenig, Ruiz, and Sogge \cite{Kenigetal1987} showed $\|(-\Delta-z/|z|)^{-1}\|_{p\to p'}\lesssim_{p,d}1$ for all $z\in\C\setminus[0,\infty)$ if and only if $q\in[d/2,(d+1)/2]$.
Thus, \eqref{eq:franklt} follows.
Let us comment on the necessity of $q\in[d/2,(d+1)/2]$. The lower bound is a consequence of Sobolev's inequality. We explain the necessity of $q\leq(d+1)/2$. By Stone's formula and the spectral theorem or Plancherel's theorem, $\im(-\Delta-1+i0)^{-1}$ is proportional to the operator $\ce\ce^*$, i.e., the composition of the Fourier restriction operator
\begin{align}
  \cs(\R^d) \ni f\mapsto (\ce^* f)(\xi) := \int_{\R^d}dx\, \me{-2\pi ix\cdot\xi}f(x)\Big|_{\xi\in\bs^{d-1}} \in L^\infty(\bs^{d-1})
\end{align}
and the Fourier extension operator
\begin{align}
  L^\infty(\bs^{d-1}) \ni g\mapsto (\ce g)(x) := \int_{\bs^{d-1}}d\sigma(\xi) \me{2\pi ix \cdot \xi}g(\xi) \in \cs'(\R^d),
\end{align}
with the Leray surface measure $\sigma$ on $\bs^{d-1}$.
In the following, we will also write $\hat{f}(\xi):=(\F f)(\xi):=\int_{\R^d}\me{-2\pi i x\cdot\xi}f(x)\,dx$ for the Fourier transform of $f$ and $(f)^{\vee}(x)=\int_{\R^d}\me{2\pi i x\cdot\xi}f(\xi)\,d\xi$ for the inverse Fourier transform.
Thus, to investigate necessary conditions for the $L^p\to L^{p'}$-boundedness of $(-\Delta-z)^{-1}$, it suffices to investigate necessary conditions for $p$ on the $L^p(\R^d)\to L^2(\bs^{d-1})$-boundedness of $\ce^*$. On the one hand, by the Riemann--Lebesgue lemma, $\|\ce^*\|_{L^1(\R^d)\to L^2(\bs^{d-1})}\lesssim1$. On the other hand, since the Fourier transform maps $L^2(\R^d)$ unitarily into itself, a Fourier restriction of $L^2(\R^d)$-functions is meaningless, since $L^2(\R^d)$-functions belong to an equivalence class of functions within which its members are allowed to differ from each other on Lebesgue measure zero sets, i.e., in particular on hypersurfaces like $\bs^{d-1}$. Considering the Fourier transform of an indicator function on a rectangular box $[0,\epsilon^{-1}]^{d-1}\times[0,\epsilon^{-2}]$ and exploiting the quadratic curvature of the sphere, one observes $p\leq2(d+1)/(d+3)$ is a necessary condition for $\ce^*$ to be $L^p(\R^d)\to L^2(\bs^{d-1})$-bounded---this is known as Knapp's example \cite{Tomas1975,Strichartz1977}. In particular, Knapp's example shows the necessity of $q\leq(d+1)/2$ in the resolvent bound of Kenig, Ruiz, and Sogge, and led to the belief that Frank's bound \eqref{eq:franklt} is optimal. Inspired by Knapp's construction, B\"ogli and the first author \cite{BogliCuenin2023} succeeded to construct potentials $V\in L^q(\R^d)$ with $q>(d+1)/2$ for which \eqref{eq:ls} is violated.

\subsection{Eigenvalue estimates for Schr\"odinger operators with random com\-plex potentials}

While the results \cite{Bogli2017,FrankSimon2017,Frank2011,BogliCuenin2023} provide a complete answer to the Laptev--Safronov conjecture~\ref{ls}, it is interesting to investigate if the range of exponents in \eqref{eq:franklt} can be increased for ``generic'' $V$.
Here, we investigate how strong one must perturb or destroy the counterexample of \cite{Bogli2017,BogliCuenin2023} to recover a Keller--Lieb--Thirring inequality.
A rather strong way to do so is to take this counterexample, or, more generally, any given deterministic potential, decomposing its support into boxes of side length $h>0$, and multiplying it on every $h$-cube with a mean-zero random variable. More precisely, for independently and identically distributed, mean-zero Gaussian or symmetric Bernoulli random variables $(\omega_j)_{j\in h\Z^d}$, we consider the Anderson randomization of $V$, given by
\begin{align}
  \label{randomization of deterministic V}
  V_{\omega}(x) = \sum_{j\in h\Z^d}\omega_jV(x)\mathbf{1}_{Q}((x-j)/h), \quad Q=[0,1)^d.
\end{align}
We denote the product measure associated to the $\omega_j$ by $\mathbf{P}$ and the expectation by $\mathbf{E}$.

Our results in \cite{CueninMerz2022}, Theorems~\ref{thm. 1}--\ref{thm. 3} below, say that Frank's eigenvalue estimate \eqref{eq:franklt} continues to hold for random potentials that are allowed to decay almost twice as slowly as their deterministic counterparts. We state this result in two forms. The first form involves compactly supported potentials with $\supp V\subseteq B(R)$.
We write $\langle x\rangle:=2+|x|$ for $x\in\R^d$, $B(x,R):=\{y\in\R^d:\,|x-y|<R\}$, and $B(R)$ for a ball with radius $R$ and unspecified center.

\begin{theorem}[{\cite[Theorem~1]{CueninMerz2022}}]
  \label{thm. 1}
  There exist constants $M_0,c>0$ such that the following holds. For any $R,\lambda>0$, $0<h<R$, $|\epsilon|\ll \lambda$, $q\leq d+1$, for any $V\in L^q(\R^d)$ with $\supp V\subseteq B(R)$, and for any $M\geq M_0$, each eigenvalue $z=(\lambda+\I\epsilon)^2$ of $-\Delta-V_{\omega}$ satisfies 
  \begin{align}
    \label{eq. thm. 1}
    \frac{\lambda^{2-\frac{d}{q}}}{\langle \lambda h\rangle^{d/2}(\ln \langle \lambda R\rangle)^{7/2}}\leq M\|V\|_{L^{q}(\R^d)},
  \end{align}
  except for $\omega$ in a set of measure at most $\exp(-cM^2)$.
\end{theorem}

While compactly supported potentials belong to any $L^q$-space, the point of \eqref{eq. thm. 1} is the weak (logarithmic) dependence on $R$. In contrast, the deterministic estimate~\eqref{eq:franklt} and H\"older's inequality yield a power law, i.e.,
\begin{align}\label{deterministic+Holder}
  \lambda^{\frac{2}{d+1}}\lesssim R^{d(\frac{2}{d+1}-\frac{1}{q})}\|V\|_{L^{q}}, \quad q\geq\frac{d+1}{2}.
\end{align}
In the next theorem, we consider $V\in L^q$, which are not assumed to be compactly supported anymore.
\begin{theorem}[{\cite[Theorem~3]{CueninMerz2022}}]
  \label{thm. 3}
  For any $q<d+1$, there exist constants $M_0,c>0$ such that the following holds. For any $h,\lambda>0$, $|\epsilon|\ll \lambda$, for any $V\in L^q(\R^d)$ and for any $M\geq M_0$, each eigenvalue $z=(\lambda+\I\epsilon)^2$ of $-\Delta-V_{\omega}$ satisfies 
  \begin{align}
    \label{eq. thm. 3}
    \frac{\lambda^{2-\frac{d}{q}}}{\langle \lambda h\rangle^{d/2}(\ln \langle \lambda h\rangle)^{2}}\leq M\|V\|_{L^{q}(\R^d)},
  \end{align}
  except for $\omega$ in a set of measure at most $\exp(-cM^2)$.
\end{theorem}


In the next Subsection, we explain the strategy of the proof of Theorems~\ref{thm. 1}--\ref{thm. 3}.
In Section~\ref{s:liebthirring}, we prove new Schatten norm estimates (Theorem~\ref{vsschattenpointwise}) and apply them to conclude bounds for eigenvalue sums of $-\Delta-V_\omega$ (Theorem~\ref{evsumsrandomschrodinger}).

\subsection{Ideas of the proof of Theorems~\ref{thm. 1}--\ref{thm. 3}}
\label{s:keyideas}

By scaling, it suffices to consider eigenvalues of the form $z=(1+i\epsilon)^2$. Moreover, by \eqref{eq:flls}, it suffices to consider $|\epsilon|\ll1$.
Again, the first step is to use the Birman--Schwinger principle.
For technical reasons, we consider the spectral radius, using Gelfand's formula
\begin{align}
  \label{eq:gelfand}
  {\rm spr}(BS(z)) = \lim_{n\to\infty}\|BS(z)^n\|^{1/n},
\end{align}
instead of $\|BS(z)\|$.
By the decay of $V$ and the uncertainty principle, the resolvents in
\begin{align}
  |V|^{1/2}R(z)VR(z)V...
\end{align}
are expected to be smoothed out in Fourier space. Indeed, by the convolution theorem, we get, for any two balls $B(x_1,R_1)$ and $B(x_2,R_2)$,
\begin{align}
  \label{eq:smoothingresolvent}
  \one_{B(x_1,R_1)}R(z)\one_{B(x_2,R_2)}
  = \one_{B(x_1,R_1)}\F^{-1}\left(\frac{1}{|\xi|^2-z} \ast \gamma_\delta\right)\F \one_{B(x_2,R_2)},
\end{align}
where $\delta^{-1}\geq R_1+R_2+|x_1-x_2|$ and $\gamma\in\cs(\R^d)$ is a Schwartz function with $\check\gamma\in C_c^\infty$.
To simplify the following notation, we denote, for $\delta>0$, by $C^{(\delta)}(D)$ any Fourier multiplier whose symbol obeys
\begin{align}
  |C^{(\delta)}(\xi)| \leq (||\xi|^2-|z||+\delta)^{-1/2}, \quad z=(1+i\epsilon)^2, \ |\epsilon|\ll1.
\end{align}
We now explain some details of the proofs of Theorems~\ref{thm. 1} and \ref{thm. 3}.

\subsubsection{On the proof of Theorem~\ref{thm. 1}}

In the situation of Theorem~\ref{thm. 1}, all resolvents appearing in Gelfand's formula for ${\rm spr}(BS(z))$ can be replaced with $C^{(\delta)}(D)^2$ with $\delta^{-1}\geq 2R$. Thus, to bound ${\rm spr}(BS(z))$, it suffices to estimate the operator norm of the ``elementary operators'' $C^{(\delta)}(D)V_\omega C^{(\delta')}(D)$ for suitable $\delta,\delta'>0$. By elliptic estimates, it suffices to consider those frequencies $|\xi|$ for which $|\xi|\sim|z|\sim1$. By the spectral theorem or Plancherel's theorem, we have, for any $F:\R_+\to\C$,
\begin{align}
  F(|D|) = \int_0^\infty d\lambda\, F(\lambda) \ce_\lambda \ce_\lambda^*,
\end{align}
where we denoted $M_\lambda:=\{\xi\in \R^d:\,|\xi|=\lambda\}$, the Fourier extension and restriction operators
\begin{align*}
  \mathcal{E}_\lambda : L^2(M_\lambda,\rd\sigma_\lambda)\to L^{\infty}(\R^d),
  \quad (\mathcal{E}_\lambda g)(x) = \int_{M_\lambda} \me{2\pi ix\cdot\xi} g(\xi) \,d\sigma_\lambda(\xi),
\end{align*}
and the Leray surface measure $\sigma_\lambda$ on $M_\lambda$. Thus, by the Cauchy--Schwarz inequality,
\begin{align}
  \label{eq:foliation}
  \begin{split}
    & \|C^{(\delta)}(D)\one_{|D|\in[1/2,2]}V_\omega C^{(\delta')}(D)\one_{|D|\in[1/2,2]}\| \\
    & \quad \lesssim \ln^{\frac12}\frac{1}{\delta} \ln^{\frac12}\frac{1}{\delta'} \sup_{\lambda,\lambda'\in[1/2,2]}\|\ce_{\lambda}^* V_\omega \ce_{\lambda'}\|_{L^2(M_{\lambda'})\to L^2(M_\lambda)},
  \end{split}
\end{align}
see also \cite[Lemma~18]{CueninMerz2022}. Thus, it suffices to estimate $\ce_{\lambda}^* V_\omega \ce_{\lambda'}$. The main technical result in \cite{CueninMerz2022} is the following

\begin{theorem}[{\cite[Lemma~15]{CueninMerz2022}}]
  \label{lemma local extension bound}
  Let $q<d+1$, $R\geq h$, and $V_\omega(x)$ be defined as in \eqref{randomization of deterministic V} with $\supp V_\omega\subseteq B(R)$. Then,
  \begin{align}
    \label{eq:lemma local extension bound}
    \sup_{\lambda,\lambda'\in[1/2,2]} \E \|\ce_{\lambda}^* V_\omega \ce_{\lambda'}\|_{L^2(M_{\lambda'})\to L^2(M_\lambda)}
    \lesssim \langle h\rangle^{\frac d2} (\ln\langle R\rangle)^{\frac12} \left(\ln\langle h\rangle + \ln\langle R\rangle\right)^2 \|V\|_{q}.
  \end{align}
  In particular, there are constants $M_0,c>0$ such that for any $M\geq M_0$, the estimate
  \begin{align}
    \label{eq:lemma local extension bound cor}
    \sup_{\lambda,\lambda'\in[1/2,2]} \|\ce_{\lambda}^* V_\omega \ce_{\lambda'}\|_{L^2(M_{\lambda'})\to L^2(M_\lambda)}
    \lesssim M\langle h\rangle^{\frac d2} (\ln\langle R\rangle)^{\frac12} \left(\ln\langle h\rangle +  \ln\langle R\rangle\right)^2 \|V\|_{q}
  \end{align}
  holds for all $\omega$ outside a set of measure at most $\me{-cM^2}$.
\end{theorem}
Estimate \eqref{eq:lemma local extension bound cor} follows from \eqref{eq:lemma local extension bound} and the Gaussian tail bound
\begin{align}
  \mathbf{P}(\|X\|>t) \leq \exp\left(-\frac{ct^2}{(\E \|X\|)^2}\right)
\end{align}
for Banach space valued random variables $X$ with symmetric Bernoulli or Gaussian distribution\footnote{We are not aware of tail bounds for Banach space valued random variables with more general sub-Gaussian distributions. This is the reason for our assumption on the distribution of the $\omega_j$ in Theorems~\ref{thm. 1}--\ref{lemma local extension bound}}, see, e.g., \cite[Lemma~12]{CueninMerz2022} and the reference \cite{LedouxTalagrand1991} therein.
Plugging \eqref{eq:lemma local extension bound cor} into \eqref{eq:foliation} and combining it with the previous strategy allows to conclude the proof of Theorem~\ref{thm. 1}.

\subsubsection{On the proof of Theorem~\ref{lemma local extension bound}}

The proof of Theorem~\ref{lemma local extension bound} uses techniques of Bourgain \cite{Bourgain2002,Bourgain2003} in the context of scattering theory of random lattice Schr\"odinger operators with long-range potentials. The basic ingredient to prove Theorem~\ref{lemma local extension bound} is the square root cancellation in the estimate
\begin{align}
  \label{eq:dudley}
  \E \sup_{j\leq N}\sum_{i=1}^j |X_i| \lesssim \sqrt{\log N} \sup_{j\leq N} \sqrt{\sum_{i=1}^j \|X_i\|_{\psi_2}^2}
\end{align}
for a sequence of symmetric Bernoulli or Gaussian random variables $\{X_i\}_{i\in\N}$.
This is stated in \cite[Section~4]{CueninMerz2022}, where we refer to \cite{Vershynin2018} for a reference of the definition of the $\psi_2$-norm and the statement and proof of \eqref{eq:dudley}.
Our goal is to apply \eqref{eq:dudley} to estimate the right-hand side of
\begin{align}
  \label{eq:taskboundopnorm}
  \E \|\ce^* V_\omega \ce\|
  = \E\sup_{g,g'\in L^2(M),\|g\|,\|g'\|\leq1}|\langle \ce g,V_\omega \ce g'\rangle|.
\end{align}
Since $L^2(M)$ has infinite cardinality, \eqref{eq:dudley} cannot be applied directly.
However, some reductions are possible. Recall that we consider a problem where both the frequencies and the positions are localized; indeed $|\xi|\sim1$ and $|x|\leq R$. By the uncertainty principle, we expect locally constant properties in position and frequency on the respective reciprocal scales, i.e., on the unit scale in position space, and on the scale $R^{-1}$ in frequency space.
In particular, we expect that we can discretize the problem in position space on the unit scale and in frequency space on the scale $R^{-1}$, and consider
\begin{align}
  \label{eq:taskboundopnormdiscrete}
  \begin{split}
    & \E\sup_{g,g'\in \ell^2(M\cap\Lambda_R^*),\|g\|,\|g'\|\leq1}|\langle \ce_{\rm d} g,v_\omega \ce_{\rm d} g'\rangle|
  \end{split}
\end{align}
instead of \eqref{eq:taskboundopnorm}, for $v_\omega=v_\omega(n)=\omega_n v_n$, $n\in h\Z^d$, discrete versions $\ce_{\rm d}^*$ and $\ce_{\rm d}$ of the Fourier restriction and extension operators, and a $1/R$-net $\Lambda_R^*\cap M$ of points on $M$.
Following Bourgain, we replace the supremum over the set $\{\ce_{\rm d}g:\,\|g\|_{\ell^2(M\cap\Lambda_R^*)}\leq1\}$ on the right-hand side of \eqref{eq:taskboundopnormdiscrete} by a supremum over a finite set up to an ``entropy'' error. Importantly, we can prove suitable bounds for the logarithm of the cardinality of this finite set, called ``entropy'' \cite[p.~170]{Vershynin2018}.
The idea to achieve this replacement is as follows.
Consider the set $\{\ce_{\rm d}g:\, \|g\|_{\ell^2(M\cap \Lambda_R^*)}\leq1\}$ and cover it with disjoint dyadic boxes of side length $2^{-k}$, $k\in\N_0$, in the $\ell^\infty(\Z^d\cap B(R))$-metric. We collect the centers of these boxes in the set $Z_k:=\{\xi_k(j)\}_{j\in\N_0}$. Then, to approximate any given $\ce_{\rm d} g\in \ell^\infty$ with $\|g\|_{\ell^2(M\cap\Lambda_R^*)}\leq1$, we start from the origin, walk to the center $\xi_0(j_0)$ of a box with side length $2^0$, which is closest to $\ce_{\rm d}g$, then walk to the center $\xi_1(j_1)$ of a box with side length $2^{-1}$, which is closest to $\ce_{\rm d}g$, etc. In this way, we can, for any $g\in \ell^2(M\cap\Lambda_R^*)$ with $\|g\|_2\leq1$, construct a chain of centers $\{\xi_k\}$ of boxes of side length $2^{-k}$ converging to $\ce_{\rm d}g$. Formally, this is just the telescoping series
\begin{align}
  \label{eq:telescope}
  \begin{split}
    \ce_{\rm d}g
    & = \ce_{\rm d}g-\xi_0(j_0)+\xi_0(j_0)
    = \ce_{\rm d}g-\xi_0(j_0)+\xi_0(j_0)-\xi_1(j_1)+\xi_1(j_1) = ... \\
    & =: \ce_{\rm d}g- \xi_0(j_0) + \sum_{k\geq0}\xi^{(k)}, \quad \xi^{(k)} \in \F_k\subseteq Z_k - Z_{k+1},
  \end{split}
\end{align}
with the property that $\|\xi^{(k)}\|_{\ell^\infty}\lesssim 2^{-k}$, $\xi^{(k)}\in\F_k$. To apply \eqref{eq:dudley}, we thus need to compute the cardinality of the subset $\F_k$ of all difference vectors between boxes with centers in $Z_k$ and $Z_{k+1}$ which are as close to each other as possible. To that end, we use the ``dual Sudakov inequality'', first proved by Pajor and Tomczak--Jaegermann \cite{PajorTomczakJaegermann1986} and later, in a simpler fashion, by Bourgain, Milman, and Lindenstrauss \cite[Proposition~4.2]{Bourgainetal1989}. It says that $\ln|Z_k| \lesssim \ln(\langle R\rangle) 4^k$. Thus,
\begin{align}
  \label{eq:dualsudakov}
  \ln|\F_k| \lesssim \ln(\langle R\rangle) 4^k.  
\end{align}
A rigorous implementation of these ideas requires an averaging over translations.
Writing $\e(x\cdot\xi):=\me{2\pi ix\cdot \xi}$, we showed in the proof of \cite[Lemma~15]{CueninMerz2022} that Formula \eqref{eq:telescope} takes the form
\begin{align}
  \sum_\nu e((x_i+y)(\eta_\nu+\tau)) g(\eta_\nu+\tau) = \sum_{k\geq0}\xi_i^{(k)}, \quad \xi^{(k)}\in\F_k,
\end{align}
with $\{x_i\}$ and $\{\eta_\nu\}$ constituting a $1$-net in $\R^d$ and a $R^{-1}$-net $\Lambda_R^*$ of $M$, respectively, $y\in B(0,10)$, $\tau,\tau'\in M\cap B(0,10/R)$, and 
\begin{align}
  \|\xi^{(k)}\|_{\ell^\infty} \lesssim 2^{-k} R^{\frac{d-1}{2}} \|g(\eta_\nu+\tau)\|_{\ell_{\nu}^2},
  \quad
  \|\xi^{(k)}\|_{\ell^{2(d+1)/(d-1)}} \lesssim R^{\frac{d-1}{2}} \|g(\eta_\nu+\tau)\|_{\ell_{\nu}^2}.
\end{align}
Thus, we can bound the right-hand side of \eqref{eq:taskboundopnorm} by
\begin{align}
  \label{eq:connectioncvx}
  \E\sup_{g,g'\in L^2(M),\|g\|,\|g'\|\leq1}|\langle \mathcal{E}^* V_{\omega} \mathcal{E}g,g'\rangle|
  \leq \sum_{k,k'\in \Z_+} \int \E \max_{(\xi,\xi')\in\mathcal{F}_k\times \mathcal{F}_{k'}}|X_{\xi,\xi'}|\,dy \,d\tau\,d\tau',
\end{align}
where the dependence of
\begin{align}
  X_{\xi,\xi'} = X_{\xi,\xi'}(y,\tau,\tau')
  := \sum_{j\in h\Z^d}\omega_j\sum_i \overline{V(x_i+y)\xi_i}\xi_i'
\end{align}
on the variables $y\in B(0,10)$, $\tau,\tau'\in M\cap B(0,10/R)$ is suppressed.
Combining \eqref{eq:dudley} with the dual Sudakov estimate \eqref{eq:dualsudakov} and applying H\"older's inequality repeatedly gives the probabilistic bound
\begin{align}
  \int \mathbf{E}\max_{(\xi,\xi')\in\mathcal{F}_k\times \mathcal{F}_{k'}}|X_{\xi,\xi'}|\,dy\,d\tau\,d\tau'
  \lesssim (\ln \langle R\rangle)^{1/2}h^{d/2} \|V\|_{L^{q}(\R^d)}.
\end{align}
On the other hand, the Riemann--Lebesgue lemma and repeated use of H\"older's inequality gives the deterministic bound
\begin{align}
  \int \max_{(\xi,\xi')\in\mathcal{F}_k\times \mathcal{F}_{k'}}|X_{\xi,\xi'}|dy\,d\tau\,d\tau'
  \lesssim R^{d/q'} 2^{-k-k'}\|V\|_{L^{q}(\R^d)}.
\end{align}
Interpolating between these two bounds yields \eqref{eq:lemma local extension bound}.

\subsubsection{On the proof of Theorem~\ref{thm. 3}}

Even if $V$ is not necessarily compactly supported anymore, we expect the decay of $V$ to smoothen out the resolvents in Gelfand's formula \eqref{eq:gelfand}.
To quantify the decay of $V$, we apply a horizontal dyadic decomposition (see, e.g., \cite[Theorem~6.6]{Tao2006Notes}), which is reminiscent to the definition of Lorentz spaces. More precisely, we write
\begin{align}
  V=\sum_{i\in\Z_+}V_i,\quad V_i=V\mathbf{1}_{H_i\geq |V|\geq H_{i+1}},\quad H_i=\inf\{t>0:\,|\{|V|>t\}|\leq 2^{i-1}\}.
\end{align}
The widths of the supports of $V_i$ are approximately $2^i$ and $\|H_i2^{i/q}\|_{\ell^r_i(\Z_+)}\sim \|V\|_{L^{q,r}}$, where $L^{q,r}$ denotes a Lorentz space. As explained in \eqref{eq:smoothingresolvent}, the smoothing of the resolvents using two spatial cut-offs like $\one_{\Omega_1}(x)R(z)\one_{\Omega_2}(x)$ depends on the diameter of the $\Omega_j$ and their relative distance. While the sizes of the $\supp(V_i)$ are known, their diameter and relative distance is not accessible yet. To control these properties, we apply a sparse decomposition to each $\supp(V_i)$. We say that for $N\in\N$, $\gamma,R>0$, a family $\{B(x_k,R)\}_{k=1}^N$ is called $\gamma$-sparse if the centers $x_k$ are $(RN)^\gamma$-separated.
For fixed $K\gg1$ and $\gamma>0$, whose values will be determined at the end of the argument, we let $K_i=\mathcal{O}(K2^{i/K})$ be the number of sparse families used to cover $\supp(V_i)$, we let $N_i=\mathcal{O}(2^i)$ be the number of balls inside any sparse family, and we let $R_i=\mathcal{O}(2^{i\gamma^K})$ be the radius of the balls.
Then, according to Tao \cite[Lemma~3.3]{Tao1999} (see also \cite[Section~2.2]{Choetal2022}),
\begin{align}
  \label{sparse decomp.}
  V_i=\sum_{j=1}^{K_i}\sum_{k=1}^{N_i}V_{ijk},
\end{align}
where, for fixed $i,j$, the $V_{ijk}$ are supported on a sparse collection of balls $\{B(x_k,R_i)\}_{k=1}^{N_i}$.
Plugging the decomposition \eqref{sparse decomp.} into Gelfand's formula \eqref{eq:gelfand} allows us to proceed similarly as in the case of compactly supported potentials. We refer to \cite[Sections~7.2--7.4]{CueninMerz2022} for details.

\section{Eigenvalue sums}
\label{s:liebthirring}

Being able to locate a region in the complex plane for which we can say where all eigenvalues of $-\Delta-V_\omega$ are located, it is natural to consider the accumulation of these eigenvalues.
This will lead us to estimate the singular values of 
\begin{align}
  \cv_{M_\lambda}^{(\omega)} :=
  \mathcal{E}_\lambda^* V_{\omega} \mathcal{E}_\lambda, \quad \lambda>0.
\end{align}
That is, we will estimate $\|\cv_M^{(\omega)}\|_{\cs^{p}}$ for some $p\geq1$, with the $p$-th Schatten norm $\|T\|_{\cs^p(\ch)}^p:=\sum_{n\geq1}s_n(T)^p$ and the singular values $\{s_n(T)\}_{n\in\N}$ of some compact operator $T$ on a Hilbert space $\ch$. We also write $\|T\|_{\cs^{p,\infty}(\ch)} := \sup_m s_m(T)m^{\frac1p}$ for the weak $p$-th Schatten norm.
For the sake of simplicity, and, because our methods are not likely to yield optimal results for $L^q$-potentials,
we will only consider compactly supported or pointwise decaying potentials in this section.

Our approach to bound $s_k(\cv_{M_\lambda}^{(\omega)})$ consists of two steps. First, we adapt the estimate \eqref{eq:lemma local extension bound} for $\E\|\cv_M^{(\omega)}\|$. Then, using an observation recorded, e.g., in \cite[(3.4.13)]{DyatlovZworski2019}, we obtain bounds for all other singular values of $\cv_M^{(\omega)}$ by essentially comparing them to the singular values of powers of the Laplace--Beltrami operator $-\Delta_M$; see Theorem~\ref{vsschattenpointwise} for the final result. As an application, we prove estimates for sums of eigenvalues of $-\Delta-V_\omega$ (Theorem~\ref{evsumsrandomschrodinger}).

\subsection{Alternative estimate for $\|\cv_M^{(\omega)}\|$}
\label{s:opnorm}

We first give an alternative bound for $\|\cv_M^{(\omega)}\|$ to that in Theorem~\ref{lemma local extension bound}, which is suitable for pointwise decaying potentials. Recall \eqref{eq:connectioncvx} with
\begin{align*}
  X_{\xi,\xi'} = X_{\xi,\xi'}(y,\tau,\tau') = \sum_{j\in h\Z^d}\omega_j\sum_i \overline{V(x_i+y)\xi_i}\xi_i',
\end{align*}
$\{x_i\}$ constituting a $1$-net in $\R^d$, $y\in B(0,10)$, and
\begin{align}
  \|\xi^{(k)}\|_{\ell^\infty} \lesssim 2^{-k} R^{\frac{d-1}{2}} \|g(\eta_\nu+\tau)\|_{\ell_{\nu}^2},
\end{align}
where $\tau\in M\cap B(0,10/R)$ and the points $\eta_\nu$ form an $1/R$-net $\Lambda_R^*$ of $M$.
For the proof of the following proposition, we also use
\begin{align}
  \label{cardinality of j}
  |\{j\in h\Z^d:\, x_i+y\in [j,j+h)^d\}| & = 1 \quad\mbox{for each } i
\end{align}
and 
\begin{align}
  \label{cardinality of i}
  |\{i:\, x_i+y\in [j,j+h)^d\}| & \leq h^d\quad\mbox{for each } j\in h\Z^d.
\end{align}

\begin{proposition}
  \label{localextensionboundptwisepotential}
  Let $R\geq h>0$ and $V\in L^\infty(\R^d)$ with $\supp V \subseteq B(R)$. Then,
  \begin{align}
    \label{eq:localextensionboundptwisepotential1a}
    & \int \E \max_{\mathcal{F}_k\times \mathcal{F}_{k'}}|X_{\xi,\xi'}|\rd y\rd\tau\rd\tau'
      \lesssim R^{1/2}(\ln \langle R\rangle)^{1/2}h^{d/2}\|V\|_\infty, \\
    \label{eq:localextensionboundptwisepotential1b}
    & \int \max_{\mathcal{F}_k\times \mathcal{F}_{k'}}|X_{\xi,\xi'}|\rd y\rd\tau\rd\tau'
      \lesssim R\, 2^{-k-k'} \|V\|_\infty.
  \end{align}
  Consequently,
  \begin{align}
    \label{eq:localextensionboundptwisepotential2}
    \begin{split}
      \E \|\cv_M^{(\omega)}\|
      \lesssim \sum_{k,k'}\int \E \max_{\mathcal{F}_k\times \mathcal{F}_{k'}}|X_{\xi,\xi'}|\rd y\rd\tau\rd\tau'
      \lesssim R^{1/2} \langle h\rangle^{d/2}(\ln \langle R\rangle)^{5/2}\|V\|_\infty,
    \end{split}
  \end{align}
  and, for all $\epsilon>0$ and $V\in \langle x\rangle^{-(1/2+\epsilon)}L^\infty$,
  \begin{align}
    \label{eq:localextensionboundptwisepotential2cor}
    \begin{split}
      \E \|\cv_M^{(\omega)}\| \lesssim_\epsilon \langle h\rangle^{d/2} \|\langle x\rangle^{\frac12+\epsilon}V\|_\infty.
    \end{split}
  \end{align}
\end{proposition}

\begin{proof}
  Formula~\eqref{eq:localextensionboundptwisepotential2cor} follows from \eqref{eq:localextensionboundptwisepotential2} by a dyadic decomposition. In turn, \eqref{eq:localextensionboundptwisepotential2} follows from \eqref{eq:localextensionboundptwisepotential1a}--\eqref{eq:localextensionboundptwisepotential1b}.
  To prove these inequalities, we follow \cite{CueninMerz2022} and assume $\|V\|_\infty=1$ without loss of generality.
  We first prove \eqref{eq:localextensionboundptwisepotential1a}. For $\ln N := \ln(\langle R\rangle) \cdot \max\{4^k,4^{k'}\}$, Formula~\eqref{eq:dudley} yields
  \begin{align}
    \E\max_{\F_k\times\F_{k'}}|X_{\xi,\xi'}|
    \lesssim \sqrt{\ln N} \left(\sum_{j\in h\Z^d} \left| \sum_i \overline{V(x_i+y) \xi_i}\, \xi_i'\right|^2 \right)^{1/2}.
  \end{align}
  By $\supp V\subseteq B(R)$ and Cauchy--Schwarz,
  \begin{align}
    \left|\sum_i \overline{V(x_i+y) \xi_i}\, \xi_i' \right|
    \lesssim \sum_{|x_i|\lesssim R}|\xi_i||\xi_i'|
    \leq \left(\sum_{|x_i|\lesssim R}|\xi_i|^2\right)^{1/2} \left(\sum_{|x_i|\lesssim R}|\xi_i'|^2\right)^{1/2}.
  \end{align}
  Thus,
  \begin{align}
    \label{eq:localextensionboundptwisepotentialaux3}
    \begin{split}
      \E\max_{\F_k\times\F_{k'}}|X_{\xi,\xi'}|
      & \lesssim \sqrt{\ln N} \left(\sum_{j\in h\Z^d} \left(\sum_{|x_i|\lesssim R}|\xi_i|^2\right) \left(\sum_{|x_i'|\lesssim R}|\xi_i'|^2\right) \right)^{1/2} \\
      & \lesssim \sqrt{\ln N} \left\|\sum_{|x_i|<R}|\xi_i|^2\right\|_{\ell_j^1}^{1/2} \left\|\sum_{|x_i|<R}|\xi_i'|^2\right\|_{\ell_j^\infty}^{1/2} \\
      & \lesssim \sqrt{\ln N} \|\xi_i\|_{\ell_i^\infty} \, \|\xi_i'\|_{\ell_i^2(|x_i|<R)} \cdot h^{d/2}.
    \end{split}
  \end{align}
  To proceed, we estimate $\|\xi_i\|_{\ell_i^2(|x_i|<R)}$.
  Recall that the $\xi_i$ are defined via $\sum_\nu e((x_i+y)(\eta_\nu+\tau))g(\eta_\nu+\tau) = \sum_{k\geq0}\xi_i^{(k)}$, i.e., they belong to the range of the discrete Fourier restriction operator.
  Let $\phi_R(\xi)=R^d\phi(R\xi)\in\cs(\R^d)$ such that $\supp\phi_R\subseteq B(1)$ and $\hat\phi\geq\one_{B(0,1)}$. Then, by the discrete Stein--Tomas theorem \cite[Proposition~1.29 and (1.18)]{Demeter2020} (see also \cite[Section~3]{CueninMerz2022}),
  \begin{align}
    \label{eq:discreterestrictionplancherel}
    \begin{split}
      & \|\sum_{\mu\in\Lambda_R^*} a(\mu) e(\mu\cdot x)\|_{L^2(|x|<R)}
      \leq \|\sum_{\mu\in\Lambda_R^*} a(\mu) \hat\phi_R(x) e(\mu\cdot x)\|_{L^2(\R^d)} \\
      & \quad = \|\sum_{\mu\in\Lambda_R^*} a(\mu) \phi_R(\xi+\nu)\|_{L^2(\R^d)}
      \leq \|a(\mu)\|_{\ell_\mu^2(\Lambda_R^*)} \|\phi_R\|_{L^2(\R^d)} \\
      & \quad \lesssim R^{d/2} \|a(\mu)\|_{\ell_\mu^2(\Lambda_R^*)}.
    \end{split}
  \end{align}
  Combining \eqref{eq:localextensionboundptwisepotentialaux3} with \eqref{eq:discreterestrictionplancherel}, we get, by symmetry between $k$ and $k'$,
  \begin{align}
    \label{eq:Xprobabilistic}
    \int_{B(0,10)} dy\, \iint_{B(0,10/R)^2} d\tau\,d\tau'\, \E\max_{\F_k\times\F_{k'}}|X_{\xi,\xi'}|
    \lesssim \sqrt{\ln \langle R\rangle} h^{d/2} R^{1/2}.
  \end{align}
  This concludes the probabilistic bound \eqref{eq:localextensionboundptwisepotential1a}.

  We now prove the deterministic bound \eqref{eq:localextensionboundptwisepotential1b}. By H\"older's inequality,
  \begin{align}
    |X_{\xi,\xi'}| \lesssim \|V(x_i+y)\|_{\ell_i^1} \|\xi\|_\infty \|\xi'\|_\infty
    \lesssim R^{2d-1} 2^{-k-k'}.
  \end{align}
  Therefore,
  \begin{align}
    \label{eq:Xdeterministic}
    \int_{B(0,10)} dy\, \iint_{(B(0,10/R)\cap M)\times (B(0,10/R)\cap M)} d\tau\,d\tau'\, |X_{\xi,\xi'}|
    \lesssim R 2^{-k-k'}.
  \end{align}
  Combining \eqref{eq:Xprobabilistic} and \eqref{eq:Xdeterministic} gives
  \begin{align}
    \label{eq:Xinterpolated}
    \int dy\,d\tau\,d\tau' \E |X_{\xi,\xi'}|
    \lesssim \min\{\sqrt{\ln \langle R\rangle} h^{d/2} R^{1/2}, R 2^{-k-k'}\}
  \end{align}
  and therefore, using \cite[Lemma~25]{CueninMerz2022},
  \begin{align}
    \begin{split}
      \E\|\ce^* V_\omega \ce\|
      & \lesssim R \sum_{k,k'}\min\{2^{-k-k'},R^{-\frac12}\ln(\langle R\rangle)^{\frac12}h^{\frac d2}\} 
        \lesssim R^{\frac12} (\ln \langle R\rangle)^{\frac52} \langle h\rangle^{\frac d2}.
    \end{split}
  \end{align}
  The proof is concluded.
\end{proof}

\subsection{Bounds for singular values of $\cv_M^{(\omega)}$}

We now prove bounds for the singular values $s_k(\cv_M^{(\omega)})$ of $\cv_M^{(\omega)}$.

\begin{theorem}
  \label{vsschattenpointwise}
  Let $R\geq h>0$ and $\nu\in(0,d-1]$. Then, for all $\lambda>0$ and $V\in L^\infty$ with $\supp V\subseteq B(R)$,
  \begin{align}
    \label{eq:vsschattenpointwisescaled}
    \E\|\ce_\lambda^* V_\omega \ce_\lambda\|_{\cs^{\frac{d-1}{\nu},\infty}}
    & \lesssim (\lambda R)^{\frac12+\nu} \sqrt{\ln \langle\lambda R\rangle} \langle \lambda h\rangle^{\frac d2} (\ln\langle \lambda R\rangle + \ln\langle \lambda h\rangle)^2 \lambda^{-d} \|V\|_\infty.
  \end{align}
  Consequently, for all $\epsilon>0$, $R_0\geq1$ and $\lambda>R_0^{-1}$,
  \begin{align}
    \label{eq:vsschattenpointwisescaledcor}
    \E\|\ce_\lambda^* V_\omega \ce_\lambda\|_{\cs^{(d-1)/\nu,\infty}}
    \lesssim_\epsilon R_0^{\epsilon} \langle \lambda h\rangle^{d/2} \lambda^{1/2+\nu+\epsilon-d} \|\langle x\rangle^{1/2+\epsilon+\nu}V\|_\infty.
  \end{align}
\end{theorem}


\begin{proof}
  Formula~\eqref{eq:vsschattenpointwisescaledcor} follows from \eqref{eq:vsschattenpointwisescaled} by a dyadic decomposition and using that $\ln(\langle \lambda R\rangle)^{5/2}\lesssim (\lambda R)^\epsilon\,R_0^\epsilon$ since $R\geq1$ without loss of generality. To prove \eqref{eq:vsschattenpointwisescaled}, it suffices to consider $\lambda=1$ by scaling; see also \cite[Proposition~4.1, Lemma~4.2]{Cuenin2022}. Using the Weyl asymptotic for the Laplace--Beltrami operator $-\Delta_M$, we get, for any $\nu>0$,
  \begin{align}
    \begin{split}
      s_k(\cv_M^{(\omega)})
      & \leq s_k(\langle -\Delta_M\rangle^{-\nu/4})^2 \cdot \|\langle -\Delta_M\rangle^{\nu/4} \ce^* V_\omega \ce \langle -\Delta_M\rangle^{\nu/4}\| \\
      & \lesssim k^{-\nu/(d-1)} \|\langle -\Delta_M\rangle^{\nu/4} \ce^* V_\omega \ce \langle -\Delta_M\rangle^{\nu/4}\|.
    \end{split}
  \end{align}
  To bound the operator norm on the right-hand side, we proceed as in Section~\ref{s:opnorm}. By an integration by parts,
  \begin{align}
    \ce\langle -\Delta_M \rangle^{\nu/4} =: \ce_{a},
  \end{align}
  where $\ce_{a}$ is a modulated Fourier extension operator, which acts as
  \begin{align}
    (\ce_{a} g)(x) = \int_M d\sigma(\xi) e(x\cdot\xi) a(x,\xi) g(\xi),
  \end{align}
  and $a(x,\xi)$ is a smooth symbol, which satisfies
  \begin{align}
    \|a(x,\cdot)\|_{C^N} \lesssim_N \langle x\rangle^{\nu/2}, \quad x\in\R^d, \ N\in\N_0.
  \end{align}
  By inspection of the proof of Proposition~\ref{localextensionboundptwisepotential}, we see that \eqref{eq:vsschattenpointwisescaled} would follow from 
  \begin{align}
    \label{eq:vsschattenpointwisescaledaux}
    \|\sum_{\mu\in\Lambda_R^*} b(\mu) \cdot a(x,\mu) \cdot e(\mu\cdot x)\|_{L^2(|x|<R)}
    \lesssim R^{d/2+\nu/2}\|b(\mu)\|_{\ell_\mu^2(\Lambda_R^*)}.
  \end{align}
  When $a(x,\xi)=c(x)d(\xi)$ factorizes, we proceed as in \eqref{eq:discreterestrictionplancherel} and obtain
  \begin{align}
    \begin{split}
      & \|\sum_{\mu\in\Lambda_R^*}b(\mu) a(x,\mu) e(\mu\cdot x)\|_{L_x^2(B(R))}
      \leq \|\sum_{\mu\in\Lambda_R^*}b(\mu) c(x)d(\mu)\hat\phi_R e(\mu\cdot x)\|_{L_x^2(B(R))} \\
      & \quad \leq \|c\|_{L^\infty(B(R))} \|\sum_{\mu\in\Lambda_R^*}b(\mu) d(\mu) \phi_R(\xi+\mu)\|_{L_\xi^2(\R^d)} \\
      & \quad \leq R^{\frac\nu2} \|b(\mu) d(\mu)\|_{\ell^2(\Lambda_R^*)} \|\phi_R\|_{L^2(\R^d)} 
        \lesssim R^{\frac{d+\nu}{2}} \|b(\mu)\|_{\ell_\mu^2(\Lambda_R^*)}\|d(\mu)\|_{\ell_\mu^\infty},
    \end{split}
  \end{align}
  as desired. For general $a(x,\mu)$, we decompose
  \begin{align}
    a(x,\xi) = \sum_{y\in\Z^d}\hat a(x,y) e(-\xi\cdot y)
  \end{align}
  into a Fourier series with
  \begin{align}
    \hat a(x,y) := 2^{-d} \int_{[-1,1]^d} a(x,\xi)e(\xi\cdot y)\,d\xi
  \end{align}
  obeying
  \begin{align}
    \sup_{x\in B(R)}|\hat a(x,y)| \lesssim R^{\nu/2} \langle y\rangle^{-N}.
  \end{align}
  Thus,
  \begin{align}
    \begin{split}
      & \|\sum_{\mu\in\Lambda_R^*}b(\mu) a(x,\mu) e(\mu\cdot x)\|_{L_x^2(B(R))} \\
      & \quad = \|\sum_{\mu\in\Lambda_R^*} \sum_{y\in\Z^d} b(\mu) e(\mu\cdot(x-y)) \hat a(x,y) \hat\phi_R \|_{L_x^2(B(R))} \\
      & \quad \leq \sum_{y\in \Z^d}\sup_{x\in B(R)}|\hat a(x,y)| \sup_{y\in\Z^d} \|\sum_{\mu\in\Lambda_R^*}b(\mu) e(-\mu\cdot y) \phi_R(\xi+\mu)\|_{L_\xi^2(\R^d)} \\
      & \quad \lesssim R^{d/2+\nu/2} \sup_{y\in\Z^d} \|b(\mu) e(-\mu\cdot y)\|_{\ell_\mu^2(\Lambda_R^*)}
        \leq R^{d/2+\nu/2}\|b(\mu)\|_{\ell^2(\Lambda_R^*)}.
    \end{split}
  \end{align}
  This concludes the proof of \eqref{eq:vsschattenpointwisescaledaux} and thereby that of \eqref{eq:vsschattenpointwisescaled} and Theorem~\ref{vsschattenpointwise}.
\end{proof}


We remark that for deterministic $V$, the proof leading to \eqref{eq:vsschattenpointwisescaled} gives
\begin{align}
  \label{eq:vsschattenpointwisedeterministic}
  \|\ce^* V \ce\|_{\cs^{(d-1)/\nu,\infty}}
  & \lesssim \|\langle x\rangle^{1+\nu+\epsilon} V\|_\infty, \quad \nu\in(0,d-1], \ \epsilon>0.
\end{align}
Up to replacing the weak with the strong Schatten norm and the additional power $\langle x\rangle^\epsilon$, this bound coincides with \cite[Proposition~8.1.3]{Yafaev2010}.

\subsection{Eigenvalue sums of random Schr\"odinger operators}

We now apply the Schatten bound in Theorem~\ref{localextensionboundptwisepotential} to estimate sums of eigenvalues of $-\Delta-V_\omega$, where $\supp V_\omega\subseteq B(R)$ for some $R>0$.
Let $\delta(z):=\dist(z,[0,\infty))$ be the distance between $z\in\C$ and the essential spectrum of $-\Delta-V_\omega$. We will use the following abstract result by Frank.

\begin{theorem}[{\cite[Theorem~3]{Frank2018E}}]
  \label{frankabstract}
  Let $\sigma>0$, $p\geq1$. Let $K(z)$, $z\in\C\setminus[0,\infty)$, be an analytic family of operators satisfying $\|K(z)\|_p \leq M|z|^{-\sigma}$ for all $z\in\C\setminus[0,\infty)$ and some $M\geq0$. Let $z_j\in \C\setminus[0,\infty)$ be the eigenvalues of $1+K$ of finite type, repeated according to their algebraic multiplicity. Then, for all $\epsilon>0$,
  \begin{align}
    \sum_{j} \delta(z_j)|z_j|^{-\frac12+\frac12(2p\sigma-1+\epsilon)_+}
    \lesssim_{\sigma,p,\epsilon} M^{\frac{1}{2\sigma}(1+(2p\sigma-1+\epsilon)_+)}.
  \end{align}
\end{theorem}

In the following, we only consider eigenvalues $z$ with $R_0^{-1}\leq |z|^{1/2} \leq h^{-1}$ for any fixed $R_0\geq1$ and $|\im z|\ll1$. The reason for the restriction $|z|^{1/2} \leq h^{-1}$ is that the randomization becomes ineffective for frequencies that are significantly larger than the reciprocal randomization length scale.
In this case, the estimate \eqref{eq:vsschattenpointwisescaled} implies for any $\epsilon>0$,
\begin{align}
  \label{eq:vsschattenpointwisescaledsimpler}
  \begin{split}
    \E\|\ce_\lambda^* V_\omega \ce_\lambda\|_{\cs^{(d-1)/(\nu-\epsilon)}}
    & \leq \E\|\ce_\lambda^* V_\omega \ce_\lambda\|_{\cs^{(d-1)/\nu,\infty}} \\
    & \lesssim_{\epsilon,\nu,d} R_0^\epsilon \lambda^{1/2+\nu+\epsilon-d}\|\langle x\rangle^{1/2+\nu+\epsilon} V\|_\infty.
  \end{split}
\end{align}
We choose $\nu=2\epsilon$.
Repeating the previous arguments that led to \eqref{eq:foliation}, we get, for any $\epsilon\in(0,d-1]$,
\begin{align}
  \label{eq:smoothedbirmanschwingerschatten}
  \begin{split}
    & \E\|C_z^{(\frac1R)}(D)V_\omega C_z^{(\frac1R)}(D)\|_{\cs^{\frac{d-1}{\epsilon}}} 
     \lesssim R_0^{\epsilon} |z|^{-3/4+3\epsilon/2} \|\langle x\rangle^{1/2+3\epsilon} V\|_\infty,
  \end{split}
\end{align}
where $C_z^{(\delta)}(D)$ is any Fourier multiplier whose symbol satisfies
\begin{align}
  \label{eq:smoothedhalfresolvent}
  |C_z^{(\delta)}(\xi)| \leq (||\xi|^2-|z||+\delta)^{-1/2}.
\end{align}
%
We now use Theorem~\ref{frankabstract} to bound eigenvalue sums of $-\Delta-V_\omega$. To that end, we use $\one_{B(R)}R(z)\one_{B(R)}=\one_{B(R)}\tilde R(z)\one_{B(R)}$, where $\tilde R(z)$ is a smoothed out resolvents with the symbol of $|\tilde R(z)|^{1/2}$ obeying the bound \eqref{eq:smoothedhalfresolvent}, and the resolvent identity
\begin{align}
  \begin{split}
    & (-\Delta-V_\omega-z)^{-1} \\
    & \quad = R(z) + R(z)V_\omega R(z) + R(z)V_\omega \tilde R(z)^{\frac12}(1-K(z))^{-1} |\tilde R(z)|^{1/2}V_\omega R(z),
  \end{split}
\end{align}
where
\begin{align}
  K(z):=|\tilde R(z)|^{\frac12}V_\omega \tilde R(z)^{\frac12}.
\end{align}
Thus, the eigenvalues $z$ of $-\Delta-V_\omega$ are those points for which $1-K(z)$ fails to be invertible.
Using Theorem~\ref{frankabstract} with
\begin{align}
  p = \frac{d-1}{\epsilon}
  \quad \text{and} \quad
  \sigma = \frac34 - \frac{3\epsilon}{2}>0,
\end{align}
and the bound \eqref{eq:smoothedbirmanschwingerschatten} for $\|K(z)\|$ yields the following

\begin{theorem}
  \label{evsumsrandomschrodinger}
  Let $\epsilon\in(0,1/2)$, $h>0$, and $V\in \langle x\rangle^{-\frac12-3\epsilon}L^\infty(\R^d)$. Let $\{z_j\}\subseteq\C\setminus[0,\infty)$ denote the set of discrete eigenvalues of finite algebraic multiplicities of $-\Delta-V_\omega$, repeated according to their algebraic multiplicity. Then, there are constants $M_0,c>0$ such that for all $R_0\geq1$, and $M\geq M_0$,
  \begin{align}
    \begin{split}
      & \sum_{R_0^{-1/2}\leq |z_j|\leq h^{-1/2}} \delta(z_j)|z_j|^{-\frac12+\frac12(\frac{3(d-1)}{2\epsilon}-3d+2+\epsilon)_+} \\
      & \quad \lesssim_{d,\nu,\epsilon} (R_0^{\epsilon} \|\langle x\rangle^{1/2+3\epsilon} V\|_\infty)^{\frac{1}{3/2-3\epsilon}(1+(\frac{3(d-1)}{2\epsilon}-3d+2+\epsilon)_+)}
    \end{split}
  \end{align}
  holds for all $\omega$ outside a set of measure at most $\me{-cM^2}$.
  In particular, for these $\omega$, there are $c_1=c_1(R_0,\epsilon,d)>0$ and $c_2=c_2(\epsilon,d)>1$ such that
  \begin{align}
    \sum_{R_0^{-1/2}\leq |z_j|\leq h^{-1/2}} \delta(z_j)
    \leq c_1 \|\langle x\rangle^{1/2+3\epsilon} V\|_\infty^{c_2}.
  \end{align}
\end{theorem}

Note that there is $\epsilon_d>0$ such that $\frac{3(d-1)}{2\epsilon}-3d+2+\epsilon>0$ if and only if $\epsilon<\epsilon_d$.

\def\cprime{$'$}


\begin{thebibliography}{10}

\bibitem{Abramovetal2001}
A.~A. Abramov, A.~Aslanyan, and E.~B. Davies.
\newblock Bounds on complex eigenvalues and resonances.
\newblock {\em J. Phys. A}, 34(1):57--72, 2001.

\bibitem{Bogli2017}
S.~B{\"o}gli.
\newblock Schr\"{o}dinger operator with non-zero accumulation points of complex
  eigenvalues.
\newblock {\em Comm. Math. Phys.}, 352(2):629--639, 2017.

\bibitem{BogliCuenin2023}
S.~B\"{o}gli and J.-C. Cuenin.
\newblock Counterexample to the {L}aptev-{S}afronov conjecture.
\newblock {\em Comm. Math. Phys.}, 398(3):1349--1370, 2023.

\bibitem{Bourgain2002}
J.~Bourgain.
\newblock On random {S}chr\"{o}dinger operators on {$\Bbb Z^2$}.
\newblock {\em Discrete Contin. Dyn. Syst.}, 8(1):1--15, 2002.

\bibitem{Bourgain2003}
J.~Bourgain.
\newblock Random lattice {S}chr\"{o}dinger operators with decaying potential:
  some higher dimensional phenomena.
\newblock In {\em Geometric Aspects of Functional Analysis}, volume 1807 of
  {\em Lecture Notes in Math.}, pages 70--98. Springer, Berlin, 2003.

\bibitem{Bourgainetal1989}
J.~Bourgain, J.~Lindenstrauss, and V.~Milman.
\newblock Approximation of zonoids by zonotopes.
\newblock {\em Acta Math.}, 162(1-2):73--141, 1989.

\bibitem{Choetal2022}
C.-H. Cho, Y.~Koh, and J.~Lee.
\newblock A global space-time estimate for dispersive operators through its
  local estimate.
\newblock {\em J. Math. Anal. Appl.}, 514(1):Paper No. 126255, 15, 2022.

\bibitem{Cuenin2022}
J.-C. {Cuenin}.
\newblock {Effective upper bounds on the number of resonance in potential
  scattering}.
\newblock {\em arXiv e-prints}, page arXiv:2209.06079, Sept. 2022.

\bibitem{CueninMerz2022}
J.-C. {Cuenin} and K.~{Merz}.
\newblock Random {S}chr{\"o}dinger operators with complex decaying potentials.
\newblock {\em arXiv e-prints}, page arXiv:2201.04466, Jan. 2022.

\bibitem{Demeter2020}
C.~Demeter.
\newblock {\em Fourier Restriction, Decoupling, and Applications}, volume 184
  of {\em Cambridge Studies in Advanced Mathematics}.
\newblock Cambridge University Press, Cambridge, 2020.

\bibitem{DyatlovZworski2019}
S.~Dyatlov and M.~Zworski.
\newblock {\em Mathematical Theory of Scattering Resonances}, volume 200 of
  {\em Graduate Studies in Mathematics}.
\newblock American Mathematical Society, Providence, RI, 2019.

\bibitem{Frank2011}
R.~L. Frank.
\newblock Eigenvalue bounds for {S}chr\"{o}dinger operators with complex
  potentials.
\newblock {\em Bull. Lond. Math. Soc.}, 43(4):745--750, 2011.

\bibitem{Frank2018E}
R.~L. Frank.
\newblock Eigenvalue bounds for {S}chr\"{o}dinger operators with complex
  potentials. {III}.
\newblock {\em Trans. Amer. Math. Soc.}, 370(1):219--240, 2018.

\bibitem{Frank2021}
R.~L. Frank.
\newblock The {L}ieb-{T}hirring inequalities: recent results and open problems.
\newblock In {\em Nine mathematical challenges---an elucidation}, volume 104 of
  {\em Proc. Sympos. Pure Math.}, pages 45--86. Amer. Math. Soc., Providence,
  RI, [2021] \copyright 2021.

\bibitem{Franketal2006}
R.~L. Frank, A.~Laptev, E.~H. Lieb, and R.~Seiringer.
\newblock Lieb-{T}hirring inequalities for {S}chr\"{o}dinger operators with
  complex-valued potentials.
\newblock {\em Lett. Math. Phys.}, 77(3):309--316, 2006.

\bibitem{FrankSimon2017}
R.~L. Frank and B.~Simon.
\newblock Eigenvalue bounds for {S}chr\"{o}dinger operators with complex
  potentials. {II}.
\newblock {\em J. Spectr. Theory}, 7(3):633--658, 2017.

\bibitem{Keller1961}
J.~B. Keller.
\newblock Lower bounds and isoperimetric inequalities for eigenvalues of the
  {S}chr\"{o}dinger equation.
\newblock {\em J. Mathematical Phys.}, 2:262--266, 1961.

\bibitem{Kenigetal1987}
C.~E. Kenig, A.~Ruiz, and C.~D. Sogge.
\newblock Uniform {S}obolev inequalities and unique continuation for second
  order constant coefficient differential operators.
\newblock {\em Duke Math. J.}, 55(2):329--347, 1987.

\bibitem{LaptevSafronov2009}
A.~Laptev and O.~Safronov.
\newblock Eigenvalue estimates for {S}chr\"{o}dinger operators with complex
  potentials.
\newblock {\em Comm. Math. Phys.}, 292(1):29--54, 2009.

\bibitem{LedouxTalagrand1991}
M.~Ledoux and M.~Talagrand.
\newblock {\em Probability in {B}anach Spaces}, volume~23 of {\em Ergebnisse
  der Mathematik und ihrer Grenzgebiete (3) [Results in Mathematics and Related
  Areas (3)]}.
\newblock Springer-Verlag, Berlin, 1991.
\newblock Isoperimetry and Processes.

\bibitem{LiebThirring1976}
E.~H. Lieb and W.~E. Thirring.
\newblock Inequalities for the moments of the eigenvalues of the
  {Schr\"o}dinger {H}amiltonian and their relation to {S}obolev inequalities.
\newblock In E.~H. Lieb, B.~Simon, and A.~S. Wightman, editors, {\em Studies in
  Mathematical Physics: Essays in Honor of {V}alentine {B}argmann}. Princeton
  University Press, Princeton, 1976.

\bibitem{PajorTomczakJaegermann1986}
A.~Pajor and N.~Tomczak-Jaegermann.
\newblock Subspaces of small codimension of finite-dimensional {B}anach spaces.
\newblock {\em Proc. Amer. Math. Soc.}, 97(4):637--642, 1986.

\bibitem{ReedSimon1978}
M.~Reed and B.~Simon.
\newblock {\em Methods of Modern Mathematical Physics}, volume 4: Analysis of
  Operators.
\newblock Academic Press, New York, 1 edition, 1978.

\bibitem{Strichartz1977}
R.~S. Strichartz.
\newblock Restrictions of {F}ourier transforms to quadratic surfaces and decay
  of solutions of wave equations.
\newblock {\em Duke Math. J.}, 44(3):705--714, 1977.

\bibitem{Tao1999}
T.~Tao.
\newblock The {B}ochner-{R}iesz conjecture implies the restriction conjecture.
\newblock {\em Duke Math. J.}, 96(2):363--375, 1999.

\bibitem{Tao2006Notes}
T.~Tao.
\newblock Lecture notes: Fourier analysis.
\newblock Available at \url{https://www.math.ucla.edu/~tao/247a.1.06f/}, 2006.

\bibitem{Tomas1975}
P.~A. Tomas.
\newblock A restriction theorem for the {F}ourier transform.
\newblock {\em Bull. Amer. Math. Soc.}, 81:477--478, 1975.

\bibitem{Vershynin2018}
R.~Vershynin.
\newblock {\em High-Dimensional Probability}, volume~47 of {\em Cambridge
  Series in Statistical and Probabilistic Mathematics}.
\newblock Cambridge University Press, Cambridge, 2018.
\newblock An Introduction with Applications in Data Science, With a foreword by
  Sara van de Geer.

\bibitem{Yafaev2010}
D.~R. Yafaev.
\newblock {\em Mathematical Scattering Theory}, volume 158 of {\em Mathematical
  Surveys and Monographs}.
\newblock American Mathematical Society, Providence, RI, 2010.
\newblock Analytic Theory.

\end{thebibliography}
\end{document}